\documentclass[12pt]{amsart}
\usepackage{amsfonts,amsthm,latexsym,amsmath,amssymb,amscd,amsmath, mathrsfs, url, cite }
\usepackage[pdftex]{color}
\usepackage[bookmarks=true,hyperindex,pdftex,colorlinks, citecolor=blue,linkcolor=blue, urlcolor=blue]{hyperref}

\usepackage[foot]{amsaddr}

\newcounter{count}
\numberwithin{count}{section}
\newtheorem{Lemma}[count]{Lemma}
\newtheorem{Remark}[count]{Remark}
\newtheorem{Example}[count]{Example}
\newtheorem{Definition}[count]{Definition}

\newtheorem{Theorem}[count]{Theorem}
\newtheorem{Conjecture}[count]{Conjecture}
\newtheorem{Problem}[count]{Open Problem}

\begin{document}

\author[O.~Katkova]{Olga Katkova}

\address{Department of Mathematics, University of Massachusetts Boston, USA}
\email{olga.katkova@umb.edu }

\author[A.~Vishnyakova]{Anna Vishnyakova$^\ast$}

\address{Department of Mathematics, Holon Institute of Technology,
Israel}
\address{$^{\ast}$Corresponding author}

\email{annalyticity@gmail.com}

\title[Convolution operators and totally positive sequences]
{Convolution operators  preserving the set of totally positive sequences}

\begin{abstract}
A real sequence  $(a_k)_{k=0}^\infty$ is called {\it totally positive}  
if all minors of the infinite Toeplitz matrix $ \left\| a_{j-i}  
\right\|_{i, j =0}^\infty$ are nonnegative (here $a_k=0$
for $k<0$). In this paper, which continues our earlier work \cite{kv}, 
we investigate the set of real sequences 
$(b_k)_{k=0}^\infty$ with the property that for every totally positive sequence 
$(a_k)_{k=0}^\infty,$ the sequense of termwise products $(a_k b_k)_{k=0}^\infty$ is 
also totally positive. In particular, we show that for every totally 
positive sequence $(a_k)_{k=0}^\infty$ the sequence 
$\left(a_k a^{-k (k-1)}\right)_{k=0}^\infty$ is totally positive whenever
$a^2\geq 3{.}503.$ We also propose several open 
problems concerning convolution operators that preserve total positivity.

\end{abstract}



\maketitle

\section{Introduction}

In this paper we study multiply positive and totally positive sequences, 
as well as linear operators that preserve these classes of sequences.

\begin{Definition}
A real sequence $(a_k)_{k=0}^\infty$  is called  \emph{$m$-times positive }
(for some $m \in \mathbb{N}$) if all minors of order at most $m$ of the
infinite matrix
\begin {equation}
\label{mat}
 \left\|
  \begin{array}{ccccc}
   a_0 & a_1 & a_2 & a_3 &\ldots \\
   0   & a_0 & a_1 & a_2 &\ldots \\
   0   &  0  & a_0 & a_1 &\ldots \\
   0   &  0  &  0  & a_0 &\ldots \\
   \vdots&\vdots&\vdots&\vdots&\ddots
  \end{array}
 \right\|
\end {equation}
are nonnegative. We denote by $\mathrm{TP}_m$ the class of
$m$-times positive sequences, and by $\widetilde{\mathrm{TP}}_m$ the class of generating functions of
$m$-times positive sequences, that is, those of the form 

\[
f(x) = \sum_{k=0}^\infty a_k x^k, \qquad (a_k)_{k= 0}^\infty \in \mathrm{TP}_m.
\]
\end{Definition}

\begin{Definition}
A real sequence
$(a_k)_{k=0}^\infty$  is called totally positive  
if all minors of the infinite matrix
(\ref{mat}) are nonnegative.
The class of totally positive sequences is denoted by $\mathrm{TP}_\infty,$
the class of the generating functions of
totally positive sequences is denoted by $\widetilde{\mathrm{TP}}_\infty.$
\end{Definition}

Multiply positive sequences (also known as P\'olya frequency
sequences) were introduced by Fekete in 1912 (see~\cite{fek}) in
connection with the problem of determining the exact number of
positive zeros of a real polynomial. 

The concept of total positivity has found numerous applications and has been
studied from many different points of view; see, for example, \cite{ando}, 
\cite{tp} or \cite{pin}.  It has applications to the distribution of zeros 
of polynomials and entire functions,  P\'olya frequency sequences, 
unimodality and log-concavity, stochastic processes and approximation 
theory, mechanical systems, planar networks, combinatorics and 
representation theory.

The class $\mathrm{TP}_\infty$ was completely described by Aissen,
Schoenberg, Whitney and Edrei in~\cite{aissen} (see also
\cite[p. 412]{tp}). We will recall their result, which has become a classic.

{\bf Theorem ASWE. } {\it A function $f $ belongs to $\widetilde{\mathrm{TP}}_\infty$ if 
and only if
\begin{equation}
\label{aswe}
f(z)=C z^q e^{\gamma z}\prod_{k=1}^\infty \frac{ (1+\alpha_kz)}{(1-\beta_kz)},
\end{equation}
where $C\ge 0, q \in \mathbb{Z},
\gamma\ge 0,\alpha_k\ge 0,\beta_k\ge 0,$ and $ \sum_{k=1}^\infty(\alpha_k+\beta_k)
<\infty.$ }

By the  ASWE theorem, a real polynomial $p(z) = \sum _{k=0}^n a_k z^k, \  a_k
\geq 0 ,$ has only real zeros if and only if  $(a_0,
a_1, \ldots , a_n, 0, 0, \ldots ) \in \mathrm{TP}_\infty .$

Determining whether a polynomial has only real zeros is typically a subtle problem. 
It arises frequently in analysis, combinatorics, approximation theory, and many other areas 
of mathematics. This motivates the study of linear operators that preserve the class of real-rooted 
polynomials. In their fundamental work on zero-preserving transformations, G.~P\'olya and 
J.~Schur (\cite{polsch}) introduced the concept of multiplier sequences.

\begin{Definition} A sequence $(\gamma_k)_{k=0}^\infty$ of real
numbers is called a multiplier sequence if, whenever a real
polynomial $P(x) = \sum_{k=0}^n a_k z^k $ has only real zeros, the
polynomial $\sum_{k=0}^n \gamma_k a_k z^k $ has only real zeros. 
We denote the class of multiplier sequences by $\mathcal{MS}$.
\end{Definition}

The following sequence gives a simple example of a multiplier sequence: $\gamma_k =k, k=0, 1, 2, \ldots$ 
Indeed, if a real polynomial $P(z) = \sum_{k=0}^n a_k z^k $ has only real zeros, the polynomial 
$\sum_{k=0}^n k a_k z^k = z P^\prime(z),$ also has only real zeros.

In their classical 1914 paper, G.~P\'olya and J.~Schur obtained a complete characterization of the set 
of  multiplier sequences. To state their theorem precisely, we recall the definition of the Laguerre-P\'olya 
class of type I.

\begin{Definition} 
A real entire function $f$ is said to be in the {\it Laguerre--P\'olya class 
of type I}, written $f \in \mathcal{L-P} I$, if it can
be expressed in the following form  
\begin{equation}  \label{lpc1}
 f(z) = c z^n e^{\beta z}\prod_{k=1}^\infty
\left(1+\frac {z}{x_k} \right),
\end{equation}
where $c \in  \mathbb{R},  \beta \geq 0, x_k >0 $, 
$n$ is a nonnegative integer,  and $\sum_{k=1}^\infty x_k^{-1} <
\infty$. The product on the right-hand side can be
finite or empty (in this case, we consider the product to be equal to 1).
\end{Definition}

Functions in the Laguerre--P\'olya class of type I play a central role in the theory of 
real-rooted polynomials and zero-preserving operators, because the class $\mathcal{L-P} I$ 
consists precisely of real entire functions that arise as locally uniform limits of polynomials 
with only real and non-positive zeros. The following famous theorem provides a stronger result. 

{\bf Theorem B} (E.~Laguerre and G.~P\'{o}lya, see, for example,
\cite[p. ~42--46]{HW}) and \cite[chapter VIII, \S 3]{lev}). {\it   

(i) Let $(P_n)_{n=1}^{\infty},\  P_n(0)=1, $ be a sequence
of real polynomials having only real negative zeros which  
converges uniformly on the disc $|z| \leq A, A > 0.$ Then this 
sequence converges locally uniformly in $\mathbb{C}$ to an entire function
 from the class $\mathcal{L-P}I.$
 
(ii) For any $f \in \mathcal{L-P}I$ there is a
sequence of real polynomials with only real nonpositive 
zeros, which converges locally uniformly to $f$.}

Now we can formulate theorem of G. P\'olya and J.Schur, which completely describes the class of 
multiplier sequences.

{\bf Theorem C} (G. P\'olya and J.Schur, cf. \cite{polsch}, \cite[pp. 100-124]{pol}
and\cite[pp. 29-47]{O}).  
{\it Let  $(\gamma_k)_{k=0}^\infty$ be a given 
real sequence. The following 
three statements are equivalent.}

{\it

1. $(\gamma_k)_{k=0}^\infty$ is a multiplier sequence.

2.  For every $n\in \mathbb{N}$ the
polynomial $P_n(z) =\sum_{k=0}^n {\binom{n}{k}} \gamma_k z^k $ has only real 
zeros of the same sign. 

3. The power series $ \Phi (z) := \sum_{k=0}^\infty \frac
{\gamma_k}{k!}z^k$ converges absolutely in the whole complex plane
and the entire function $\Phi(z)$ or the entire function
$\Phi(-z)$ admits the representation
\begin{equation}
\label{pred}
 c  z^n e^{\beta z} \prod_{k=1}^\infty
\left(1+\frac{z}{x_k}\right),
\end{equation}

where $c\in{\mathbb{R}}, \beta \geq 0, n\in {\mathbb{N}}\cup \{0\}, 0<x_k
\leq \infty,\   \sum_{k=1}^\infty \frac{1}{x_k} < \infty.$}

The following fact is a simple corollary of the previous result.

{\bf Corollary of Theorem C.}   {\it The sequence 
$(\gamma_0, \gamma_1, \ldots, \gamma_l,$ $ 0, 0, \ldots )$ is a
multiplier sequence if and only if the polynomial $P(z)=
\sum_{k=0}^l \frac {\gamma_k}{k!}z^k$ has only real zeros of the
same sign.}

\begin{Remark}
It is easy to see that, if $(\gamma_k)_{k=0}^\infty$ is a multiplier sequence, then
for every $l\in \mathbb{N}$ the sequence $(\gamma_{k+l})_{k=0}^\infty$ is also 
a multiplier sequence. The situation for totally positive sequences is far subtler. It 
turned out that $\mathrm{TP}_\infty$  is not shift-invariant under coefficients, even 
when the original sequence is extremally well behaved.  

In order to show that, let us consider a sequence  ${\mathbf E} =\left(\frac{1}{k!}\right)_{k=0}^\infty$
with the generating function $E(z)=\sum_{k=0}^\infty  \frac{z^k}{k!} = e^z.$ 
By theorem ASWE, the sequence ${\mathbf E} $ is totally positive. At the same time, the sequence 
$\left(\frac{1}{(k+1)!}\right)_{k=0}^\infty$ is not totally positive. Indeed, its generating function 
$E_1(z)=\sum_{k=0}^\infty  \frac{z^k}{(k+1)!} = \frac{e^z -1}{z}$ has infinitely many non-real zeros. 
Hence, it cannot correspond to a totally positive sequence by the ASWE characterization. 

One can prove that for every $l\in \mathbb{N}$ the sequence
$\left(\frac{1}{(k+l)!}\right)_{k=0}^\infty$ is not totally positive.
\end{Remark}

In this paper, we study an analog of the multiplier sequences for the set of 
totally positive sequences. We need the following definition.

\begin{Definition}  Let ${\mathbf A} =(a_k)_{k=0}^\infty$ be  a nonnegative sequence.
We define the following linear convolution operator on the set of
real sequences:
$$\Lambda_{\mathbf A} ( (b_k)_{k=0}^\infty) =  (a_k b_k)_{k=0}^\infty. $$
\end{Definition} 

The following problem was posed by Alan Sokal during the 
 AIM workshop ``Theory and applications of total positivity'',
July 24-28, 2023 (see \cite{conf} for more details). 

\begin{Problem} \label{P1}  To describe the set of non-negative sequences 
${\mathbf A} =(a_k)_{k=0}^\infty,$ such that the corresponding convolution 
operator $\Lambda_{\mathbf A}$ preserves the set of $\mathrm{TP}_\infty-$sequences:
for every  $(b_k)_{k=0}^\infty \in \mathrm{TP}_\infty$ we have $\Lambda_{\mathbf A} 
((b_k)_{k=0}^\infty) \in  \mathrm{TP}_\infty. $ From now on such sequences ${\mathbf A}$ will be called 
$\mathrm{TP_\infty}-$preservers.
\end{Problem}

See \cite{DyaSok} by A.~Dyachenko and A.~Sokal (and a previous works of 
A.~Dyachenko \cite{Dya1} and \cite{Dya2}) in connection with the problem 
above.

We denote by $A$ the generating function of a sequence 
 ${\mathbf A} =(a_k)_{k=0}^\infty,$ that is the function
\[
A(z) =\sum_{k=0}^\infty a_k z^k.\]
 
A natural candidate for a $\mathrm{TP}_\infty-$preserver is the multiplier sequence 
${\mathbf \Gamma} = (k)_{k=0}^\infty,$ since multiplication by $k$ corresponds to 
differentiation, which preserves real-rootedness. Indeed, 
the corresponding convolution operator $\Lambda_{\mathbf \Gamma} 
( (b_k)_{k=0}^\infty) =  (k b_k)_{k=0}^\infty$ preserves the set of finite
totally positive sequences (the set of coefficients of polynomials with non-negative coefficients 
and only real zeros). However, this analogy with multiplier sequences breaks down in the infinite 
setting: this operator does not preserve $\mathrm{TP}_\infty$.
In order to see that, let us consider the function $f(z) = \frac{1}{(1-z)(2-z)}= \sum_{k=0}^\infty
c_k z^k,$ where $c_k =1 - \frac{1}{2^{k+1}}$. By Theorem ASWE, $ (c_k)_{k=0}^\infty
\in \mathrm{TP}_\infty.$ But $\sum_{k=0}^\infty k c_k z^k = z f^\prime(z) =
 \frac{z(3 - 2z)}{(1-z)^2 (2-z)^2},$ this function has a positive zero, so the sequence
 of its coefficients is not a $\mathrm{TP}_\infty-$sequence. Thus, differentiating does not 
 preserve  the set of generating functions of all  totally positive sequences. 

This failure of differentiation to preserve $\mathrm{TP}_\infty$ motivates the following 
fundamental question: for which generating functions of totally positive sequences 
$B(z) =\sum_{k=0}^\infty b_k z^k$ does its derivative $B'(z)$ remain in the class 
$\widetilde{\mathrm{TP}}_\infty?$ Our first result provides a complete answer to this 
question.
 
\begin{Theorem} \label{St1} Let  $(b_k)_{k=0}^\infty$   be  a $ \mathrm{TP}_\infty-$sequence. 
Then $(k b_k)_{k=0}^\infty \in \mathrm{TP}_\infty$ if and only if the generating function 
$B(z) =\sum_{k=0}^\infty b_k z^k$ is an entire function from the class  $\mathcal{L-P} I$ or 
$B(z) = \frac{P(z)}{(1-\beta z)^m},$ where $ \beta > 0, m \in \mathbb{N},$  $n = \deg P \leq m,$
and $P$ has non-negative coefficients and only real non-positive zeros.
\end{Theorem}

Since Theorem~\ref{St1} characterizes for which sequences  multiplication by $k$ preserves 
total positivity,  it seems natural to study a similar question for other important multiplier 
sequences. It is well known (see, for example, \cite{cc1},  \cite{klv1} or \cite{ngthv3}), that the 
following sequences are multiplier sequences, i.e. preserve the set of totally positive sequences 
with entire generating functions:
\begin{eqnarray}
\label{ms11}
\nonumber
 \left(\frac{1}{k!}\right)_{k=0}^\infty, \   \left(\frac{1}{a^{k(k-1)}}\right)_{k=0}^\infty, a>1, \quad  
\left(\frac{1}{k! a^{k(k-1)}}\right)_{k=0}^\infty, a>1, 
& \\  
\left(\frac{1}{(q^k-1)(q^{k-1} -1)\cdot \ldots \cdot (q-1)}\right)_{k=0}^\infty, q>1.
\end{eqnarray}   
The following open problem is of interest.             

\begin{Problem}
Let $(a_k)_{k=0}^\infty$ be one of the sequences from the list (\ref{ms11}). 
Describe the set of totally positive sequences $(b_k)_{k=0}^\infty,$
such that the sequence $(a_k b_k)_{k=0}^\infty$ is also totally positive.
\end{Problem}
A partial answer to this problem for the case $a_k =\frac{1}{a^{k(k-1)}}, a>1, 
k=0, 1, 2,  \ldots$  is given in the Theorem~\ref{Th3}. 
 
In our previous work \cite{kv}, we completely characterized the $TP_\infty$-preservers whose 
generating functions have at least one pole. We recall that result below.

{\bf Theorem D} (\cite{kv}).  
{\it Let ${\mathbf A} =(a_k)_{k=0}^\infty$ 
be a nonnegative sequence, and its generating function is a meromorphic function 
with at least one pole. Then  for every  $(b_k)_{k=0}^\infty \in \mathrm{TP}_\infty$ we have  
$\Lambda_{\mathbf A} ((b_k)_{k=0}^\infty) \in  \mathrm{TP}_\infty$  if and only if $A(z) = 
\frac{C}{1- \beta z}, C > 0, \beta > 0.$}

Thus, in the meromorphic case,  $\mathrm{TP}_\infty$-preservers are precisely geometric sequences. 
The main difficulties arise in the case of entire generating functions. The description of 
$\mathrm{TP}_\infty$-preservers whose generating functions are entire 
functions is still unknown.

To analyze the case where the generating function is entire, it is convenient to introduce the second 
quotients, which provide a useful measure of the local convexity of the coefficient sequence.

 Let  $f(z) = \sum_{k=0}^\infty a_k z^k$  be an 
entire function with positive coefficients. We define the second quotients  $q_n$ as follows:

\begin{equation}
\label{qqq} 
q_n=q_n(f):=\frac {a_{n-1}^2}{a_{n-2}a_n},\
n\geq 2.
\end{equation}

The formulas below follow by iteratively applying the definition of the second quotients.

\begin{equation}
\label{defq}
 a_n=\frac
{a_1}{q_2^{n-1} q_3^{n-2} \ldots q_{n-1}^2 q_n} \left(
\frac{a_1}{a_0} \right) ^{n-1},\ n\geq 2.
\end{equation}

The second quotients $q_n$ provide useful information about the distribution of zeros 
of the associated entire function. This connection dates back to classical work of 
J.I. Hutchinson~\cite{hut}. In 1926, Hutchinson established one of the first general 
coefficient-based criteria for an entire functions  with positive coefficients and all its 
Taylor sections to have only real non-positive zeros.  His result shows that large second 
quotients force the function into the Laguerre--P\'olya class of type I.

{\bf Theorem E} (J. I. ~Hutchinson, \cite{hut}). { \it Let $f(x)=
\sum_{k=0}^\infty a_k x^k$, $a_k > 0$ for all $k$. 
Then 
$$q_n(f)\geq 4 \ \ \mbox{ for all} \ \ n\geq 2,$$  
if and only if the following two conditions are fulfilled:

(i) The zeros of $f(x)$ are all real, simple and negative, and 

(ii) the zeros of any polynomial $\sum_{k=m}^n a_k x^k$, $m < n,$  formed 
by taking any number 
of consecutive terms of $f(x) $, are all real and non-positive.}

From Hutchinson's theorem we see that $f$ has only real zeros when $q_k(f) \in  [4, + \infty), 
k \geq 2.$ It is easy to show that, if the estimation of $q_k(f)$ only from below is
given then the constant  $4$  in $q_k (f) \geq 4$ is the smallest possible to conclude
that $f  \in \mathcal{L-P}I$. However, if we have the estimation of $q_k$ from below and  
from above, then the constant $4$ in the condition $q_k \geq 4$ can be reduced to 
conclude that $f \in \mathcal{L-P}I$, see  \cite{HishAn}. For some extensions of Hutchinson's
results see \cite[\S4]{cc1}. See also \cite{ngthv2} for some necessary conditions.

The following example was given by Alan Sokal.
\begin{Example} \label{E1}  
Let us consider an entire function of the form $f(z) = \sum_{k=0}^\infty
a_k z^k$ with $a_0 = a_1 =1,$ $a_k =\frac{1}{q_2^{k-1}q_3^{k-2}
\cdot \ldots \cdot q_{k-1}^2 q_{k}}$ for $k\geq 2,$ where
$(q_k)_{k=2}^\infty$ are arbitrary parameters 
under the following conditions: $q_k \geq 4$ for all $k.$
Suppose that  $(b_k)_{k=0}^\infty$   is  any $ \mathrm{TP}_\infty-$sequence. 
Since every $\mathrm{TP}_\infty-$sequence is, in particular, $2-$times positive sequence, we have
$\frac{b_{n-1}^2}{b_{n-2}b_n}\geq 1.$
So, for an entire function $ (A\ast B) (z) =\sum_{k=0}^\infty 
a_k b_k z^k $ the second quotient $\frac{(a_{n-1} b_{n-1})^2}{(a_{n-2} b_{n-2)}
(a_n b_n)} = \frac{a_{n-1}^2}{a_{n-2}a_n} \cdot 
\frac{b_{n-1}^2}{b_{n-2}b_n}\geq 4$ for all $n\geq 2.$  
Thus, using Theorem~E by Hutchinson, we get
$ (A\ast B)(z) \in \mathrm{TP}_\infty.   $
\end{Example}

Suppose that $A(z) =\sum_{k=0}^\infty a_k z^k$  is a generating function of a 
$\mathrm{TP}_\infty$-preserver.  For a given $l \in \mathbb{N}\cup \{0\}$ we consider 
the sequence ${\mathbf B} = (b_k)_{k=0}^\infty$  such that $b_0=b_1=\cdots=b_{l-1}=0$ 
and $b_k =1, \ k\geq l.$ Then by Theorem ASWE its generating function $B(z)=
\sum_{k=l}^{\infty}z^k=\frac{z^l}{1-z} \in \widetilde{\mathrm{TP}}_\infty.$ We conclude that
$\left(A\ast B\right)(z)=\sum_{k=l}^{\infty} a_k z^k \in \widetilde{\mathrm{TP}}_\infty.   $
Thus, the following statement is valid.

 \begin{Theorem} If $A(z) =\sum_{k=0}^\infty a_k z^k$  is a 
generating function of a $\mathrm{TP}_\infty$-preserver, then for every 
$l \in \mathbb{N}\cup\{0\}$ we have
$$R_l[A](z) =  \sum_{k=l}^{\infty} a_k z^k \in \widetilde{\mathrm{TP}}_\infty.$$
\end{Theorem}

In \cite{kv} the following conjecture was formulated.

\begin{Conjecture} \label{C1}  
 Let ${\mathbf A} =(a_k)_{k=0}^\infty$ be a non-negative sequence whose 
 generating function is entire. Then this sequence is a $\mathrm{TP}_\infty$-preserver if and only 
if for every $l \in \mathbb{N} \cup \{0\}$ a formal power series $\sum_{k=l}^\infty
a_k z^k$ is an entire function from the $ \mathcal{L-P} I$ class (in particular, it has only 
real nonpositive zeros). 
\end{Conjecture}

In \cite{kv} this conjecture was proved for the cases when $A$ is a polynomial of 
degree $2, 3$ or $4.$

We note that entire functions  whose Taylor sections have only real zeros have
been studied in various works (see, for example, \cite{kosshap} or \cite{klv1}), 
while entire functions whose remainders have only real zeros have been studied 
less (some results could be found in the very interesting survey \cite{iv}).
We mention here a way to construct such a function. 

\begin{Definition}
The entire function $g_a(z) =\sum _{j=0}^{\infty}  \frac{z^j}{ a^{j(j-1)}}$, $a>1,$
is called the \textit{partial theta-function}. (Note that sometimes the function 
$g_a(\frac{z}{a}) =\sum _{j=0}^{\infty} 
 \frac{z^j}{ a^{j^2}}$, $a>1,$  is called the partial theta function).
\end{Definition}
The function $g_a$ has an interesting property: for all $ n=2, 3, \ldots$ its second 
quotients $q_n(g_a)=a^2.$

The  survey \cite{War} by S.~O.~Warnaar 
contains the history of investigation of the partial theta-function and some of its 
main properties.  

The paper \cite{klv} answers to the question:  for which 
$a>1$ the function $g_a$ belongs to the class $\mathcal{L-P} I. $ In particular,
in \cite{klv} the following statement is proved.

{\bf Theorem F} (O.M.~Katkova, T.~Lobova, A.M.~Vishnyakova, \cite{klv}). { \it  
There exists a constant $q_\infty \approx 3{.}23363666 \ldots,  $
such that $g_a \in \mathcal{L-P} I $ if and only if $a^2 \geq q_\infty.$  }

In \cite{ngthv1} the following theorem is proved. Let  $f(z) = \sum_{k=0}^\infty
a_k z^k$ with $a_0 = a_1 =1,$ $a_k =\frac{1}{q_2^{k-1}q_3^{k-2}
\cdot \ldots \cdot q_{k-1}^2 q_{k}}$ for $k\geq 2,$ where $(q_k)_{k=2}^\infty$ 
is a sequence of arbitrary parameters under the conditions:
$q_2 \geq q_3 \geq q_4 \geq \ldots,$ and $\lim_{n\to\infty}q_n \geq q_\infty.$
Then $f \in \mathcal{L-P} I.$ Using this theorem, we conclude that such entire
function $f$ has all remainders with only real zeros.

Our next result is the following theorem.

\begin{Theorem} \label{Th1}
 Let $f(z) = \sum_{k=0}^\infty a_k z^k$ be an entire function
with positive coefficients such that for every $l \in \mathbb{N}\cup \{0\} $ 
the remainder $R_{l}[f]= \sum_{k=l}^\infty
a_k z^k$ belongs to the $  \mathcal{L-P} I-$class. 
Then $ q_n(f) > 3$ for all  $  n \geq 2.$
\end{Theorem}

We need further the following theorem by I.V.~Ostrovskii.

{\bf Theorem G} (I.V.~Ostrovskii, \cite{iv1}). {\it Let $f(z)=\sum_{k=0}^{\infty} a_k z^k$ 
be a power series with infinite radius 
of convergence. If there exist two nonnegative integers $n_1$ and $n_2, $ 
$n_1 \ne n_2,$ such that the remainders $R_{n_1}[f](z)$ and $R_{n_2}[f](z)$ 
have only real non-positive zeros, then
\begin{equation}
\label{est} \log M(r,f)=O(\sqrt{r}),  \ r\to \infty,
\end{equation}
where, as usual, $M(r,f)= \max_{|z| \leq r} |f(z)|, r\geq 0.$ }

The estimation (\ref{est}) implies that the order of growth of a function $f$ is not greater
than $\frac{1}{2}.$ For a function with order of growth less than $1$ the following 
statements are equivalent: that a function has only real nonpositive zeros, and that 
a function belongs to the class  $ \mathcal{L-P} I. $  It follows 
from Ostrovskii's theorem that if a power series $f$ and all its remainders $R_{l}[f],\ l=1,2,\ldots,$ 
have only real non-positive zeros, then $R_{l}[f]\in  \mathcal{L-P} I, \ l=0,1,2\ldots$  Thus, our theorem 
can be equivalently reformulated as follows.

\begin{Theorem} \label{Th2} 
 Let $f(z) = \sum_{k=0}^\infty a_k z^k$ be an entire function
 such that for every $l \in \mathbb{N}\cup \{0\} $ the function $R_{l}[f]= \sum_{k=l}^\infty
a_k z^k$ have only real non-positive zeros.
Then  $ q_n(f)  > 3, \  n=  2, 3,  \ldots.$
\end{Theorem}

In connection with the above theorem, we formulate the following open problem.

\begin{Problem}  
Let us consider the set $\mathcal{R}$ of all entire functions  
$f(z) = \sum_{k=0}^\infty a_k z^k$  such that for every $l \in \mathbb{N}\cup \{0\} $ 
the function $R_l[f](z)=\sum_{k=l}^\infty a_k z^k$ has only real non-positive zeros. 
Find the following constant 
$$c = \inf \left\{ q_n(f)   \left|  \right. f\in \mathcal{R},
   n= 2, 3, \ldots \right\}.$$
\end{Problem}

\begin{Conjecture} 
In the previous problem $c=q_\infty.$
\end{Conjecture}

Our last result is the following theorem.

\begin{Theorem} \label{Th3} 
Let $f(z)=1 + z+\sum_{k=2}^\infty  \frac{z^k}{q_2^{k-1}q_3^{k-2}\cdots q_{k}} 
\in \widetilde{\mathrm{TP}}_\infty.$ Then for each $a,  a^2 \geq 3{.}503,$ the function $$\left(f \ast g_a \right)(z)
=1+z+\sum_{k=2}^\infty  \frac{z^k}{a^{k(k-1)} q_2^{k-1}q_3^{k-2}\cdots q_{k}} \in \widetilde{\mathrm{TP}}_\infty.$$
\end{Theorem}

\section{Proof of Theorem~\ref{St1} }

Let us prove necessity. Since $(b_k)_{k=0}^\infty \in \mathrm{TP}_\infty,$ we have by theorem ASWE 
$$
B(z) = \sum_{k=0}^\infty b_k z^k =C z^q e^{\gamma z}\prod_{k=1}^\infty \frac{ (1+\alpha_kz)}{(1-\beta_kz)},
$$
where $C\ge 0, q \in \mathbb{Z},
\gamma\ge 0,\alpha_k\ge 0,\beta_k\ge 0, \sum_{k=1}^\infty(\alpha_k+\beta_k)
<\infty.$
If $B$ is entire function then the necessary condition is fulfilled. Now let $B$ is not an entire function,
so it has poles. Suppose that  $B $  has at least 2 different positive poles. Since $B$ does not 
have positive zeros, then $B^\prime$ has a positive zero (between poles), that is impossible. Thus,
$B$  has one (maybe, multiple) pole, and we have 
\begin{equation}
\label{e1+}
B(z) =C z^q e^{\gamma  z}  \frac{\prod_{k=1}^\infty(1+\alpha_k z)}{(1-\beta z)^m}
=: \frac{F(z)}{(1-\beta z)^m} ,
\end{equation}
were $C >0, q\in \mathbb{N} \cup \{0\}, \gamma \geq 0,$ $\alpha_k \geq 0, \beta > 0,$
$\sum_{k=1}^\infty \alpha_k  < \infty. $

Since $F$ is an entire function with non-negative Taylor coefficients,
we have $M(r, F) = \max_{|z| \leq r} |F(z)| =F(r).$ If $F$ is not a nonnegative constant,
then $\lim_{x\to +\infty} F(x) = + \infty.$ If $\lim_{x\to +\infty} 
\frac{F(x)}{(1-\beta x)^m} = + \infty,$ then $B^\prime$ has a positive zero
on $(\beta, +\infty),$ that is impossible. We conclude that
\begin{equation}
\label{e2+}
B(z) =   \frac{C z^q \prod_{k=1}^r(1+\alpha_k z)}{(1-\beta z)^m}
=: \frac{P(z)}{(1-\beta z)^m} ,
\end{equation}
were $C >0, q\in \mathbb{N} \cup \{0\}, $ $\alpha_k \geq 0, \beta > 0,$
$r \in \mathbb{N}\cup \{0\}, q + r\leq m. $
The necessity is proved.

Let us prove sufficiency. If $B(z) = \sum_{k=0}^\infty b_k z^k $ is an entire
function, then by theorem ASWE this function belongs to the  class  
$ \mathcal{L-P} I,$ and thus its derivative $z B^\prime (z) = 
\sum_{k=0}^\infty k b_k z^k $ also belongs to the  class  $ \mathcal{L-P} I.$ 
Thus, in this case the sufficient condition is fulfilled.\\ 

Suppose that
$B(z) = (1-\beta  z)^{-m}P(z),$ where $ \beta > 0,$ $m \in \mathbb{N},$ the polynomial 
$P(z)=\sum_{k=0}^{n} c_k z^k ,\  c_k\geq 0,$ has only non-positive zeros, and $\deg P \leq m.$ Then 

\begin{equation}
\label{B}
B^\prime(z) = \frac{P^\prime(z) (1-\beta  z)+m \beta P(z)}{(1-\beta  z)^{m+1}} 
=: \frac{Q(z)}{ (1-\beta  z)^{m+1}}.
\end{equation}
Obviously, 
\begin{equation}
\label{Q}
Q(z)=\beta(m-n)c_n z^n+\sum_{k=0}^{n-1}((k+1)c_{k+1}+\beta (m-k) c_k) z^k
\end{equation}
is a polynomial with non-negative coefficients, and $\deg Q \leq m.$ Since all zeros of $P$ are 
non-positive while $\beta>0,$ due to Rolle's Theorem $Q$ has $n-1$ zeros in the interval between 
the smallest and the biggest zero of $P,$ that is, $Q$ has $n-1$ non-positive zeros. In the case when 
$m=n,$ by virtue of (\ref{Q}) we have $\deg Q=n-1.$ Thus, in this case all zeros of $Q$ are non-positive. 
If $m>n,$ then 
\[
\lim_{x\to -\infty}B(z)=\lim_{x\to -\infty}\frac{P(z)}{(1-\beta z)^m}=0
\]
Therefore, $Q$ has one more negative zero in the infinite interval whose right endpoint is the smallest root 
of $P$. We have proved that all zeros of $Q$ are real and non-positive.
Theorem~\ref{St1} is proved.  $\Box$

\section{Proof of Theorems~\ref{Th1} and~\ref{Th2}  }

We need the following lemma from \cite{ngthv2}.

\begin{Lemma}(\cite[Lemma~2.1]{ngthv2}) \label{ll1}
 Let $f, \ f(0)>0,$ be an entire function from $ \mathcal{L-P} I.$ Then $q_2 \geq 2$ and
\begin{equation}
\label{inequality}
q_3(q_2-4)+3\geq 0.
\end{equation}
\end{Lemma}
For the reader's convenience we give the proof.
\begin{proof}
Since $f(z)=\sum_{k=0}^{\infty} \in \mathcal{L-P} I$ and $f(0)>0,$ according to the definition (\ref{lpc1}) 
the function $f$ admits the representation
\begin{equation}
\label{lpc1n}
f(z)=c e^{\beta z} \prod_{k=1}^{\infty}\left(1+\frac{z}{x_k}\right),
\end{equation}
where $c>0,$ $\beta\geq 0,$ $x_k>0,$ and $\sum x_k^{-1}< \infty.$
From the above representation we obtain the following formulas:
\begin{equation}
\label{In1}
\beta +\sum_{k=1}^{\infty}\frac{1}{x_k}=\frac{f'(0)}{f(0)}=\frac{a_1}{a_0}
\end{equation}
\begin{equation}
\label{In2}
\sum_{k=1}^{\infty}\frac{1}{x_k^2}=\frac{(f'(0))^2-f''(0)f(0)}{(f(0))^2}=\frac{a_1^2-2a_2 a_0}{a_0^2}
\end{equation}
and
$$\sum_{k=1}^{\infty}\frac{1}{x_k^3}=\frac{1}{2}\left(\frac{f'''(0)f^2-3f''(0)f'(0)f(0)+2(f'(0))^3}{(f(0))^3}\right)=$$
\begin{equation}
\label{In3}
\frac{3a_3 a_0^2-3a_2a_1a_0+ a_1^3}{a_0^3}
\end{equation}
It follows from (\ref{In2}) that $q_2=\frac{a_1^2}{a_0 a_2}\geq 2.$

According to the Cauchy-Bunyakovsky-Schwarz inequality, we obtain
$$\left(\frac{1}{x_1} + \frac{1}{x_2} + ...\right) \left(\frac{1}{x_1^3} + \frac{1}{x_2^3} + 
\ldots\right) \geq \left(\frac{1}{x_1^2 }+ \frac{1}{x_2^2} + \ldots\right)^2.$$
Therefore, by (\ref{In1}), (\ref{In2}) and (\ref{In3}) we have
$$\left(\frac{a_1}{a_0}-\beta\right)\left(\frac{3a_3 a_0^2-3a_2a_1a_0+ a_1^3}{a_0^3}\right)\geq 
\left(  \frac{a_1^2-2a_2 a_0}{a_0^2} \right)^2$$
Because $\beta \geq0,$ we obtain the inequality
$$\left(\frac{a_1}{a_0}\right)\left(\frac{3a_3 a_0^2-3a_2a_1a_0+ a_1^3}{a_0^3}\right)\geq 
\left(  \frac{a_1^2-2a_2 a_0}{a_0^2} \right)^2$$
or
$$a_1^2a_2  - 4a_0a_2^2+ 3a_0 a_1a_3 \geq 0.  $$

Since $q_2=\frac{a_1^2}{a_0 a_2},$ $q_3=\frac{a_2^2}{a_1 a_3},$ and $q_2 q_3=\frac{a_1 a_2}{a_0 a_3},$ 
the above inequality is equivalent to the following one 
$$\frac{a_1 a_2}{a_0 a_3}-4\frac{a_2^2}{a_1 a_3}+3=q_3(q_2 - 4) + 3 \geq 0.$$
Lemma~\ref{ll1} is proved.  $\Box$

\end{proof}

Assume that $q_2 \leq 3.$ In this case we can rewrite the inequality 
(\ref{inequality}) in the form
\begin{equation}
\label{inequality1}
q_3 \leq \frac{3}{4-q_2}.
\end{equation}
Note that if $q_2\leq 3,$ then 
\begin{equation}
\label{inequality2}
\frac{3}{4-q_2}\leq q_2.
\end{equation}
(The inequality (\ref{inequality2}) is equivalent to $q_2^2-4q_2+3\leq 0$). 

It follows from (\ref{inequality1}) and (\ref{inequality2}) that
\begin{equation}
\label{basic}
q_3\leq q_2 \leq 3.
\end{equation}

Since all  remainders  $R_{l}[f]$ belong to $\mathcal{L-P} I$ for all 
$l \in \mathbb{N}\cup \{0\}, $    Lemma~\ref{ll1} can be rewritten as follows: 
\begin{equation}
\label{inequalitynew}
q_{l+2}(q_{l+1}-4)+3\geq 0
\end{equation}
for all $l=1,2,3,\ldots,$ and the inequality (\ref{basic})  provides the following chain of inequalities
\begin{equation}
\label{chain}
\cdots \leq q_s\leq q_{s-1}\leq \cdots\leq  q_3\leq q_2 \leq 3.
\end{equation}

Therefore, the sequence $(q_n)_{n=2}^\infty$ is decreasing and bounded from below, that is 
the sequence is convergent. Denote by 
$$q=\lim_{l\to \infty} q_l.$$
Taking the limit in (\ref{inequalitynew}) as $n$ approaches infinity, we obtain
$$q(q-4)+3 \geq 0,$$
and, using $q\geq 2,$ we get
$$q\geq 3.$$
By virtue of (\ref{chain}), the last inequality is possible if and only if $q_s(f)=3$ for all $s=2,3,\ldots.$ 
In this case, we obtain a partial theta-function
$$ g_{\sqrt{3}}(z)=\sum_{k=0}^{\infty}\frac{z^k}{(\sqrt{3})^{k(k-1)}}.$$
By \cite[Theorem~F]{klv},  this function doesn't belong to $\mathcal{L-P} I.$
We came to a contradiction. Thus, $q_2>3.$ Because all remainders belong to $\mathcal{L-P} I,$ we have
$q_l>3$ for all $  l=2, 3, \ldots.$

Theorems~\ref{Th1} and~\ref{Th2} are proved. $\Box$

\section{Proof of Theorem~\ref{Th3} }

Suppose $f(z) = \sum_{k=0}^\infty a_k z^k$ is an entire function
with positive coefficients such that for every $l \in \mathbb{N}\cup \{0\} $ 
the remainder $R_{l}[f]= \sum_{k=l}^\infty a_k z^k$ belongs to the 
$  \mathcal{L-P} I-$ class.  Without loss of generality, we can assume 
that $a_0=a_1=1,$ since we can 
consider a function $g(z) =a_0^{-1} f (a_0 a_1^{-1}z) $  instead of 
$f,$ due to the fact that such rescaling of $f$ preserves the property of 
having all real zeros for a function and all its remainders, and preserves the 
second quotients:  $q_n(g) =q_n(f)$ 
for all $n.$ During the proof we use notation  $q_n$ instead of 
 $q_n(f).$ So we can write $ f(z) = 1 + z + \sum_{k=2}^\infty 
\frac{ z^k}{q_2^{k-1} q_3^{k-2} \ldots q_{k-1}^2 q_k}.$ We will also consider
a function  $\varphi(z) = f(-z) = 1 - z + \sum_{k=2}^\infty  \frac{ (-1)^k z^k}
{q_2^{k-1} q_3^{k-2} \ldots q_{k-1}^2 q_k}$  instead of $f.$

Let us denote by 
\begin{equation}
\label{a1}
S_n(x, a)=\sum_{k=0}^n (-1)^k \frac{x^k}{a^{k(k-1)}},
\end{equation}
and by
\begin{equation}
\label{a2}
S_n^q(x, a)=1-x+\sum_{k=2}^n (-1)^k \frac{x^k}{a^{k(k-1)} 
q_2^{k-1}q_3^{k-2}\cdots q_{k}}. 
\end{equation}

The following lemma can be proved by direct calculation.

\begin{Lemma} (see \cite[Lemma 2.4]{HishAn}, c.f. \cite[Lemma 2]{klv}). \label{ll2}  
If $a^2 \geq 1+\sqrt{5},$ then $S_4(x,a)$ has two real zeros 
in the interval $(1, a^2).$
\end{Lemma}

\begin{Lemma}  \label{ll3}
Assume that $a^2 \geq 1+\sqrt{5}.$ If $x_0 $ is the smallest root 
of the equation
\begin{equation}
\label{point}
\frac{x}{a^3}+\frac{a^3}{x}=\frac{a^3}{2},
\end{equation}
then $S_4(x_0,a)\leq 0,$ and $x_0\in (1, a^2).$ 
\end{Lemma}

\begin{proof}  Consider a polynomial
$$P_a(t)=S_4(a^3t,a)=1-a^3t+a^4t^2-a^3t^3+t^4=$$ 
$$t^2\left((t+t^{-1})^2-a^3(t+t^{-1})+a^4-2\right).$$
If $t+t^{-1}=\frac{x}{a^3}+\frac{a^3}{x}=\frac{a^3}{2}$, 
then
$$(t+t^{-1})^2-a^3(t+t^{-1})+a^4-2=-\frac{1}{4}\left(a^6-4a^4+8\right)=$$
\begin{equation}
\label{point5}
-\frac{1}{4}(a^2-2)(a^4-2a^2-4)\leq 0
\end{equation}
for all  $a$ such that $a^2 \geq 1+\sqrt{5}.$

Since $x_0$ is the smallest root of the equation (\ref{point}),
\begin{equation}
\label{point1}
x_0=\frac{a^6}{4}\left(1-\sqrt{1-\frac{16}{a^6}}\right).
\end{equation}

Show that $x_0\in (1, a^2).$ By (\ref{point1}), the point $x_0\in (1, a^2)$ if and only if
\begin{equation}
\label{point2}
1-\frac{4}{a^4} < \sqrt{1-\frac{16}{a^6}}<1-\frac{4}{a^6}.
\end{equation}
The inequality $\sqrt{1-\frac{16}{a^6}}<1-\frac{4}{a^6}$ is obvious.
Let us prove that 
\begin{equation}
\label{pointequ}
1-\frac{4}{a^4} < \sqrt{1-\frac{16}{a^6}}.
\end{equation}
This inequality is equivalent to 
$a^4-2a^2-2>0,$ which is true because $a^2 \geq 1+\sqrt{5}.$
Lemma~\ref{ll3} is proved. 
\end{proof}

\begin{Lemma} \label{ll4}
 Let $f(z)=1+z+\sum_{k=2}^\infty  \frac{z^k}{q_2^{k-1}q_3^{k-2}\cdots 
 q_{k}} \in \widetilde{TP}_\infty.$ Then
\begin{equation}
\label{D2}
\Delta_2^k =:1-\frac{1}{q_k}\geq 0,  \quad  k=2, 3, \ldots
\end{equation}
\begin{equation}
\label{D3}
D_3 =:1-\frac{2}{q_2}+\frac{1}{q_2^2 q_3}\geq 0,  
\end{equation}
\begin{equation}
\label{D4}
D_4 =:1-\frac{3}{q_2}+\frac{2}{q_2^2 q_3}+\frac{1}{q_2^2}-
\frac{1}{q_2^3 q_3^2 q_4}\geq 0,  
\end{equation}
\begin{equation}
\label{D3k}
\Delta_3^k=:1-\frac{2}{q_k}+\frac{1}{q_k^2}\left( \frac{1}{q_{k-1}}+\frac{1}{q_{k+1}} 
\right)-\frac{1}{q_{k-1}q_k^2 q_{k+1}}\geq 0,  
\end{equation}
for $k=2, 3, \ldots$
\end{Lemma}

\begin{proof} The inequalities (\ref{D3}),  (\ref{D4}) and (\ref{D2}) for $k=2$ follow 
from non-negativity of principal minors of the matrix
\begin {equation}
\label{mat1}
 \left\|
  \begin{array}{cccc}
   1 & \frac{1}{q_2} & \frac{1}{q_2^2 q_3} &\frac{1}{q_2^3 q_3^2 q_4} \\
   1   & 1 & \frac{1}{q_2} &\frac{1}{q_2^2 q_3} \\
   0   &  1  & 1 & \frac{1}{q_2}\\
   0   &  0  & 1  & 1 \\
     \end{array}
 \right\|.
\end {equation}

In order to prove  (\ref{D3k}) we consider the minor
\begin{equation}
\label{mat2}
 \det \left\|
  \begin{array}{ccc}
   \frac{1}{q_2^{k-2}q_3^{k-3}\cdots q_{k-1}} & \frac{1}{q_2^{k-1}q_3^{k-2}
   \cdots q_{k}} & \frac{1}{q_2^{k}q_3^{k-1}\cdots q_{k+1}}  \\
   \frac{1}{q_2^{k-3}q_3^{k-4}\cdots q_{k-2}}  & \frac{1}{q_2^{k-2}q_3^{k-3}
   \cdots q_{k-1}} & \frac{1}{q_2^{k-1}q_3^{k-2}\cdots q_{k}} \\
   \frac{1}{q_2^{k-4}q_3^{k-5}\cdots q_{k-3}}   &  \frac{1}{q_2^{k-3}q_3^{k-4}
   \cdots q_{k-2}}  & \frac{1}{q_2^{k-2}q_3^{k-3}\cdots q_{k-1}}\\
    \end{array}
 \right\|.
\end {equation}
If we divide each row of the determinant above by its first entry, we will see that it is non-negative 
if and only if the following determinant is nonnegative.
\begin{equation}
\label{mat3}
 \det \left\|
  \begin{array}{ccc}
   1 & \frac{1}{q_2 q_3\cdots q_{k}} & \frac{1}{q_2^{2}q_3^{2}
   \cdots q_k^2 q_{k+1}}  \\
   1 & \frac{1}{q_2 q_3 \cdots q_{k-1}} & \frac{1}{q_2^{2}q_3^{2}
   \cdots q_{k-1}^2 q_{k}} \\
  1   &  \frac{1}{q_2 q_3\cdots q_{k-2}}  & \frac{1}{q_2^{2}q_3^{2}
  \cdots q_{k-2}^2 q_{k-1}}\\
    \end{array}
 \right\|.
\end {equation}
Let us divide each column of the determinant by its last entry. We conclude that 
the sign of this determinant is the same as that of the following determinant
\begin{equation}
\label{mat4}
 \det \left\|
  \begin{array}{ccc}
   1 & \frac{1}{q_{k-1}q_{k}} & \frac{1}{q_{k-1}q_{k}^2 q_{k+1}}  \\
   1 & \frac{1}{q_{k-1}} & \frac{1}{q_{k-1}q_{k}} \\
  1   &  1  & 1\\
    \end{array}
 \right\|=\frac{1}{q_{k-1}}\Delta_3^k.
\end {equation}
Note that $\Delta_2^{k-1}$ is a minor of the matrix in the left side of the formula (\ref{mat4}).

Lemma~\ref{ll4} is proved. 
\end{proof}

\begin{Lemma} \label{ll5} 
The following inequality is valid
$$ \frac{a^2}{x^2}\left(S_4(x,a)-S_4^q (x,a)\right)> \frac{x^2}{a^{10}} D_4+ 
\left(\frac{x}{a^4}-2\frac{x^2}{a^{10}} \right) D_3 +\left(1-2\frac{x}{a^4}+
2\frac{x^2}{a^{10}} \right)\Delta_2^2. $$
\end{Lemma}

\begin{proof} It follows from (\ref{a1}) and (\ref{a2}) that

$$\frac{a^2}{x^2}\left(S_4(x,a)-S_4^q (x,a)\right)=$$
\begin{equation}
\label{T}
\left(1-\frac{1}{q_2} \right)-\frac{x}{a^4}\left(1-\frac{1}{q_2^2 q_3} \right)+
\frac{x^2}{a^{10}}\left(1-\frac{1}{q_2^3 q_3^2 q_4} \right)=:T_4^q(x,a).
\end{equation}

By (\ref{D4}) 
$$T_4^q(x,a)-\frac{x^2}{a^{10}}D_4=$$
$$\left(1-\frac{1}{q_2} \right)-\frac{x}{a^4}\left(1-\frac{1}{q_2^2 q_3} \right)
+\frac{3}{q_2} \cdot \frac{x^2}{a^{10}}-\frac{1}{q_2^2} \cdot 
\frac{x^2}{a^{10}}-\frac{2}{q_2^2 q_3} \cdot \frac{x^2}{a^{10}}>$$
\begin{equation}
\label{TD4}
\left(1-\frac{x}{a^4}\right)+\frac{1}{q_2}\left(-1+2\frac{x^2}{a^{10}}\right)+
\frac{1}{q_2^2 q_3}\left(\frac{x}{a^4}-2\frac{x^2}{a^{10}}\right).
\end{equation}

Hence,
$$T_4^q(x,a)-\frac{x^2}{a^{10}}D_4-\frac{1}{q_2^2 q_3}\left(\frac{x}{a^4}-
2\frac{x^2}{a^{10}}\right)>$$
$$ \left(1-2\frac{x}{a^4}+2\frac{x^2}{a^{10}}\right)-\frac{1}{q_2}\left(1-2\frac{x}{a^4}+
2\frac{x^2}{a^{10}}\right)=$$
\begin{equation}
\label{TD3}
\left(1-2\frac{x}{a^4}+2\frac{x^2}{a^{10}}\right)\Delta_2^2
\end{equation}
Lemma~\ref{ll5}  is proved. 
\end{proof}
\begin{Lemma} \label{ll6} 
Let $x_0$ be defined by (\ref{point}). Then
\begin{equation}
\label{e1}
\frac{x_0}{a^4}-2\frac{x_0^2}{a^{10}}> 0,
\end{equation}
and
\begin{equation}
\label{e2}
1-2\frac{x_0}{a^4}+2\frac{x_0^2}{a^{10}}>0.
\end{equation}
\end{Lemma}
\begin{proof}  By virtue of (\ref{point1}) we have
$$\frac{x_0}{a^4}-2\frac{x_0^2}{a^{10}}=\frac{x_0}{a^4}\left(1-2\frac{x_0}{a^6}\right)=
\frac{x_0}{2a^4}\left(1+\sqrt{1-\frac{16}{a^6}}\right)>0.$$
The inequality (\ref{e1}) is proved.

Let us prove (\ref{e2}). Denote by $t=\frac{a^6}{x_0}.$ We have
\begin{equation}
\label{e3}
1-2\frac{x_0}{a^4}+2\frac{x_0^2}{a^{10}}=\frac{x_0^2}{a^{10}}
\left(\frac{1}{a^2}t^2-2t+2\right).
\end{equation}
Therefore, (\ref{e2}) is equivalent to the inequality
\begin{equation}
\label{e4}
\frac{1}{a^2}t^2-2t+2>0,
\end{equation}
which is true if 
\begin{equation}
\label{e5}
t>a^2\left(1+\sqrt{1-\frac{2}{a^2}}\right).
\end{equation}
By (\ref{point1})
$$t=\frac{a^6}{x_0}=\frac{a^6}{4}\left(1+\sqrt{1-\frac{16}{a^6}}\right).$$
So, (\ref{e5}) can be rewritten in the form
\begin{equation}
\label{e6}
\frac{a^4}{4}\left(1+\sqrt{1-\frac{16}{a^6}}\right)>1+\sqrt{1-\frac{2}{a^2}}.
\end{equation}
Since $a^2>3$, the following inequalities are valid
$$\frac{a^4}{4}>1 \quad \text{and} \quad 1-\frac{16}{a^6}>1-\frac{2}{a^2}.$$
Hence, (\ref{e6}) is true.

Lemma~\ref{ll6}  is proved. 
\end{proof}

\begin{Theorem}\label{inn1}
 If $a^2 \geq 1+\sqrt{5},$ then $S_{n}^q((x_0,a))<0, \  n=4,5,\ldots,$ 
and $(f\ast g_a)((x_0,a))<0.$
\end{Theorem}
\begin{proof}
It follows from Lemma~\ref{ll5}, Lemma~\ref{ll6} and Lemma~\ref{ll3}
 that
\begin{equation}
\label{monoton}
S_4^q(x_0,a)<S_4(x_0,a) \leq 0.
\end{equation}
Let's notice that 
\begin{equation}
\label{monoton1}
\frac{x^{k+1}}{a^{(k+1)k}q_2^{k}q_3^{k-1}\cdots q_{k+1}}<
\frac{x^{k}}{a^{k(k-1)}q_2^{k-1}q_3^{k-2}\cdots q_{k}}
\end{equation}
if and only if
\begin{equation}
\label{monoton2}
x<a^{2k}q_2q_3\cdots q_{k+1}.
\end{equation}
Obviously, by Lemma~\ref{ll3}, the number $x_0$ satisfies (\ref{monoton2}) for all $k=4,5,\ldots$ 
Thus, by  (\ref{monoton}) and  (\ref{monoton1})
$$(f\ast g_a)((x_0,a))=S_{4}^q((x_0,a))-$$ $$\sum_{k=2}^{\infty} 
\left(\frac{x_0^{2k+1}}{a^{(2k+1)(2k)}q_2^{2k}q_3^{2k-1}\cdots 
q_{2k+1}}-\frac{x_0^{2k+2}}{a^{(2k+2)(2k+1)}q_2^{2k+1}q_3^{2k}
\cdots q_{2k+2}}\right)\leq $$
$$ S_{4}^q((x_0,a))<0.$$ 
Using the same idea, we can show that
$$S_{n}^q((x_0,a))<S_{4}^q((x_0,a))<0$$
for all $n=5,6,\ldots$
Theorem~\ref{inn1} is proved. 
\end{proof}

\begin{Lemma} \label{ll7} 
Let $\hat{x}_m=a^{(2m-3)}q_2 q_3\cdots q_{m-1}\sqrt{q_m}.$ For all 
$m=3, 4, \ldots, n-1 $ the following estimation is valid
$$(-1)^{(m-1)}a^{(m-1)(m-2)} q_2^{m-2}q_3^{m-3}\cdots q_{m-1} 
\frac{S_{n}^q((\hat{x}_m,a))}{\hat{x}_m^{m-1}}\geq$$
\begin{equation}
\label{L6}
 1-\frac{2}{a \sqrt{q_m}}+\frac{1}{a^4 q_m}\left(\frac{1}{q_{m+1}}+
 \frac{1}{q_{m-1}}\right) -\frac{1}{a^9 q_m^{3/2}}\left(\frac{1}{q_{m+1}}+
 \frac{1}{q_{m-1}}\right)^2.
\end{equation}
\end{Lemma}
\begin{proof} Note that by virtue of (\ref{monoton1}) and (\ref{monoton2})  
for $\hat{x}_m=a^{(2m-3)}q_2 q_3\cdots q_{m-1}\sqrt{q_m}$ we have
$$(-1)^{(m-1)} S_{n}^q((\hat{x}_m,a))\geq \frac{\hat{x}_m^{m-1}}{a^{(m-1)(m-2)} 
q_2^{m-2}q_3^{m-3}\cdots q_{m-1}}$$  
\begin{align*}
\label{b1}
\left(1-\frac{\hat{x}_m}{a^{2m-2} q_2 q_3\cdots q_{m}}-\frac{a^{2m-4} q_2 q_3
\cdots q_{m-1}}{\hat{x}_m}+\frac{\hat{x}_m^2}{a^{4m-2} q_2^{2}\cdots 
q_ m^2 q_{m+1}}+\right. \\
\left. \frac{a^{4m-10} q_2^2 \cdots q_{m-2}^2q_{m-1}}{\hat{x}_m^2}
-\frac{\hat{x}_m^3}{a^{6m} q_2^3\cdots q_m^3 q_{m+1}^2 q_{m+2}}-
\frac{a^{6m-18} q_2^3 \cdots q_{m-3}^3 q_{m-2}^2q_{m-1}}{\hat{x}_m^3}\right)
\end{align*}
(in the case of $m=3$ we will assume that $q_1=1$).
Hence, substituting an explicit expression for $\hat{x}$ into the above estimation we obtain
$$(-1)^{(m-1)}a^{(m-1)(m-2)} q_2^{m-2}q_3^{m-3}\cdots q_{m-1} 
\frac{S_{n}^q((\hat{x}_m,a))}{\hat{x}_m^{m-1}}\geq $$
$$1-\frac{2}{a \sqrt{q_m}}+\frac{1}{a^4 q_m}\left(\frac{1}{q_{m+1}}+
\frac{1}{q_{m-1}}\right) -\frac{1}{a^9 q_m^{3/2}}\left(\frac{1}{q_{m+1}^2 q_{m+2}}+
\frac{1}{q_{m-1}^2 q_{m-2}}\right)$$
\begin{equation}
\label{b2}
\geq 1-\frac{2}{a \sqrt{q_m}}+\frac{1}{a^4 q_m}\left(\frac{1}{q_{m+1}}+
\frac{1}{q_{m-1}}\right) -\frac{1}{a^9 q_m^{3/2}}\left(\frac{1}{q_{m+1}^2}+
\frac{1}{q_{m-1}^2}\right).
\end{equation}
The inequality (\ref{L6}) follows immediately from  (\ref{b2}).
Lemma~\ref{ll7}   is proved. 
\end{proof}

It follows from Lemma~\ref{ll7} that 
\begin{equation}
\label{b3}
\text{If} \ \ q_m\geq \frac{4}{a^2}, \ \text{then}\ \ (-1)^{m-1} 
S_{n}^q((\hat{x}_m,a))>0, \ \ m=3,4,\ldots, n-1.
\end{equation}

\begin{Lemma} \label{ll8}
When $1\leq q_m<\frac{4}{a^2},$ the inequality below is true
$$1-\frac{2}{a \sqrt{q_m}}+\frac{1}{a^4 q_m}\left(\frac{1}{q_{m+1}}+
\frac{1}{q_{m-1}}\right) -\frac{1}{a^9 q_m^{3/2}}\left(\frac{1}{q_{m+1}}+
\frac{1}{q_{m-1}}\right)^2\geq$$
\begin{equation}
\label{b4}
1-\frac{2}{a}+\frac{1}{a^4 }\left(\frac{1}{q_{m+1}}+\frac{1}{q_{m-1}}\right) -
\frac{1}{a^9} \left(\frac{1}{q_{m+1}}+\frac{1}{q_{m-1}}\right)^2.
\end{equation}
\end{Lemma}
\begin{proof} Consider the function
$$\varphi(t)=1 \ - \ \frac{2}{a}t \ \ +$$
\begin{equation}
\label{b5}
\frac{1}{a^4 }\left(\frac{1}{q_{m+1}}+\frac{1}{q_{m-1}}\right) t^2-
\frac{1}{a^9} \left(\frac{1}{q_{m+1}}+\frac{1}{q_{m-1}}\right)^2 t^3, \   
\frac{a}{2}\leq  t \leq 1.
\end{equation}
Let us find the derivative $\varphi^\prime (t):$
\begin{equation}
\label{b6}
\varphi^\prime (t)=-\frac{2}{a}+\frac{2}{a^4 }\left(\frac{1}{q_{m+1}}+
\frac{1}{q_{m-1}}\right) t-\frac{3}{a^9} \left(\frac{1}{q_{m+1}}+
\frac{1}{q_{m-1}}\right)^2 t^2.  
\end{equation}
The discriminant $\mathcal{D}$ of the quadratic polynomial above 

$$\mathcal{D}=\frac{4}{a^{10}}\left(a^2-6 \right) \left(\frac{1}{q_{m+1}}+
\frac{1}{q_{m-1}}\right)^2<0$$

when $a^2<4.$
Hence, $\varphi^\prime (t)<0,$ and for all $\frac{a}{2}\leq  t \leq 1$ we have
$$\varphi(t)\geq \varphi(1).$$
The statement of Lemma~\ref{ll8} follows from the last estimation if we
put $t=\frac{1}{\sqrt q_m}$. 
Lemma~\ref{ll8} is proved. 
\end{proof}

\begin{Lemma} \label{ll9} If $3<a^2<4$, and $1\leq q_m<\frac{4}{a^2},$ 
the following estimation is valid 
\begin{equation}
\label{b7}
\left(q_{m-1}-\frac{1}{A}\right)\left(q_{m+1}-\frac{1}{A}\right)\leq \frac{1-A}{A^2},
\end{equation}
where
\begin{equation}
\label{b8}
A=\frac{8(a^2-2)}{a^4}.
\end{equation} 
\end{Lemma}
\begin{proof} Let us rewrite (\ref{D3k}) in the form
\begin{equation}
\label{b9}
q_{m-1}q_{m+1}(q_{m}^2-2 q_{m})+q_{m-1}+q_{m+1}-1\geq 0.
\end{equation}
Note that for $1\leq q_m<\frac{4}{a^2},$ and $3<a^2<4,$ the following is true
$$q_{m}^2-2 q_{m}\  <\frac{16-8a^2}{a^4}=-A \ <0.$$
Thus, by (\ref{b9})
\begin{equation}
\label{b10}
-A q_{m-1}q_{m+1}+q_{m-1}+q_{m+1}-1>0.
\end{equation}
The inequality (\ref{b7}) follows from (\ref{b10}).
Lemma~\ref{ll9} is proved. 
\end{proof}

Because $1<\frac{1}{q_{m+1}}+\frac{1}{q_{m-1}}\leq 2,$  and $ a^2>,3$ 
the right part of the inequality (\ref{b4}) increases as a function of 
$\frac{1}{q_{m+1}}+\frac{1}{q_{m-1}}.$ Denote by 
\begin{equation}
\label{b11}
\mu=\min \left\{
\left. \frac{1}{q_{m+1}}+\frac{1}{q_{m-1}}\right|    
\left(q_{m-1}-\frac{1}{A}\right)\left(q_{m+1}-\frac{1}{A}\right)\leq \frac{1-A}{A^2}\right\},
\end{equation}
where $A$ is defined by (\ref{b8}), and $3<a^2<4.$ It follows from (\ref{b4}) that
$$1-\frac{2}{a \sqrt{q_m}}+\frac{1}{a^4 q_m}\left(\frac{1}{q_{m+1}}+
\frac{1}{q_{m-1}}\right) -\frac{1}{a^9 q_m^{3/2}}\left(\frac{1}{q_{m+1}}+
\frac{1}{q_{m-1}}\right)^2\geq$$
\begin{equation}
\label{b12}
1-\frac{2}{a}+\frac{1}{a^4 }\mu -\frac{1}{a^9} \mu^2.
\end{equation}

\begin{Lemma} \label{ll10} 
If $3<a^2<4$, and $1\leq q_m<\frac{4}{a^2},$ then 
\begin{equation}
\label{b13}
\mu \geq A,
\end{equation}
and
$$1-\frac{2}{a \sqrt{q_m}}+\frac{1}{a^4 q_m}\left(\frac{1}{q_{m+1}}+
\frac{1}{q_{m-1}}\right) -\frac{1}{a^9 q_m^{3/2}}\left(\frac{1}{q_{m+1}}+
\frac{1}{q_{m-1}}\right)^2\geq$$
\begin{equation}
\label{b14}
1-\frac{2}{a}+\frac{8}{a^6}-\frac{16}{a^8}-\frac{8}{a^{13}}+
\frac{32}{a^{15}}-\frac{32}{a^{17}}.
\end{equation}
\end{Lemma}
\begin{proof} 
Note that if either $q_{m-1}\leq \frac{1}{A},$ or $q_{m+1}\leq \frac{1}{A},$ the 
statement (\ref{b13}) is obvious.
Assume that both $q_{m-1}> \frac{1}{A}$ and $q_{m+1}> \frac{1}{A}.$ Denote by
$$ \mathfrak{m}=\min\{q_{m-1}, q_{m+1}\} \ \ \text{and} \ \ q=\max\{q_{m-1}, q_{m+1}\}.$$
Now (\ref{b7}) can be rewritten in the form
\begin{equation}
\label{b15}
\left(\mathfrak{m}-\frac{1}{A}\right)\left(q-\frac{1}{A}\right) \leq \frac{1-A}{A^2}.
\end{equation}
By virtue of (\ref{b7})
$$\left( \mathfrak{m}-\frac{1}{A}\right)^2 \leq \left(q_{m-1}-\frac{1}{A}\right)
\left(q_{m+1}-\frac{1}{A}\right) \leq \frac{1-A}{A^2}.$$
Therefore,
\begin{equation}
\label{b16}
\frac{1}{A}<\mathfrak{m}\leq \frac{1+\sqrt{1-A}}{A}.
\end{equation}
Let us fix $\mathfrak{m}:$ $\frac{1}{A}<\mathfrak{m}\leq \frac{1+\sqrt{1-A}}{A}.$ 
Since $\mathfrak{m}$ is fixed, and $\frac{1}{\mathfrak{m}}+\frac{1}{q}$ is decreasing 
in $q,$ by virtue of (\ref{b15}) the smallest possible value of 
$\frac{1}{\mathfrak{m}}+\frac{1}{q}$ is attained when 
\begin{equation}
\label{b17}
\left(\mathfrak{m}-\frac{1}{A}\right)\left(q(\mathfrak{m})-\frac{1}{A}\right) = \frac{1-A}{A^2},
\end{equation}
that is when
\begin{equation}
\label{b18}
q(\mathfrak{m})=\frac{\mathfrak{m}-1}{A\mathfrak{m}-1}.
\end{equation}
Consider the function 
\begin{equation}
\label{b19}
\psi(\mathfrak{m})=\frac{1}{\mathfrak{m}}+\frac{1}{q(\mathfrak{m})}=
\frac{1}{\mathfrak{m}}+\frac{A\mathfrak{m}-1}{\mathfrak{m}-1},
\end{equation}
and find its minimum on the interval $\left[\frac{1}{A},\frac{1+\sqrt{1-A}}{A} \right].$\\
We have
$$\psi ^\prime(\mathfrak{m})=\frac{-A \mathfrak{m}^2+2\mathfrak{m}-1}{\mathfrak{m}^2 
(\mathfrak{m}-1)^2}.$$
Since $3<a^2<4,$ by (\ref{b8}) we have $\frac{8}{9}<A<1,$ whence the quadratic polynomial 
in the numerator of $\psi ^\prime(\mathfrak{m})$ has two real roots: $ \frac{1-\sqrt{1-A}}{A} $ 
and $\frac{ 1+\sqrt{1-A}}{A}.$ Hence, $\psi ^\prime(\mathfrak{m})>0$ on the interval 
$\left(\frac{1}{A},\frac{1+\sqrt{1-A}}{A} \right).$ Thus, 
\begin{equation}
\label{b20}
\psi(\mathfrak{m}) \geq \psi\left(\frac{1}{A}\right)=A.
\end{equation}

The estimation (\ref{b13}) follows immediately from (\ref{b20}). Now the estimation 
(\ref{b14}) can be derived from (\ref{b12}), (\ref{b13}) and (\ref{b8}).
Lemma~\ref{ll10} is proved. 
\end{proof}
It follows from Lemma~\ref{ll7}  and (\ref{b14}) that in the case, when 
$q_m<\frac{4}{a^2},$ we have
\begin{equation}
\label{b20a}
 (-1)^{m-1} S_{n}^q((\hat{x}_m,a))\geq 0,     m=3,4,\ldots, n-1,   
 \Leftrightarrow   a^2 \geq a_0^2 \approx 3{.}41089186.
\end{equation}

The formulas (\ref{b3}) and (\ref{b20a}) provide the following result.

\begin{Theorem}\label{ins2} There exists a constant $a_0\approx \sqrt{3.41089186}$ such 
that for all $a\geq a_0$ and for all $m=3,4, \ldots, n-1,$ the estimation is valid
$$ (-1)^{m-1} S_{n}^q(\hat{x}_m,a)\geq 0,$$
where $$\hat{x}_m=a^{(2m-3)}q_2 q_3\cdot \ldots \cdot q_{m-1}\sqrt{q_m}.$$
\end{Theorem}

\begin{Lemma}\label{ll11} For $\hat{x}_{n}=a^{(2n-3)}q_2 q_3\cdots q_{n-1}\sqrt{q_n}$ 
the following estimation is valid
$$(-1)^{(n-1)}a^{(n-1)(n-2)} q_2^{n-2}q_3^{n-3}\cdots q_{n-1} 
\frac{S_{n}^q(\hat{x}_n,a)}{\hat{x}_n^{n-1}}\geq$$
\begin{equation}
\label{L10}
 1-\frac{2}{a \sqrt{q_n}}+\frac{1}{a^4 q_n q_{n-1}} -
 \frac{1}{a^9 q_n^{3/2}q_{n-1}^2}.
\end{equation}
\end{Lemma}
The proof of Lemma~\ref{ll11} is similar to the proof of Lemma~\ref{ll7} .

\begin{Theorem} \label{ins3} Let $ \sqrt{3.411}<a<2.$ Assume that there exists a 
sequence of positive integer numbers $(n_k)_{k=1}^{\infty}$ such that $ q_{n_k} 
\geq \frac{4}{a^2}.$ Then $f \ast g_a \in \mathcal{LP }I.$
\end{Theorem}

\begin{proof} By virtue of (\ref{L10}), (\ref{monoton}) and (\ref{monoton1})
\begin{equation}
\label{L11}
 (-1)^{n_k-1}S_{n_k}^q(a^{2n_k-3}q_2\ldots q_{n_k-1} \sqrt{q_{n_k}},a) >0,
\end{equation}
\begin{equation}
\label{L12}
 S_{n_k}^q(1,a)) >0,
\end{equation}
and
\begin{equation}
\label{L13}
 (-1)^{n_k}S_{n_k}^q(a^{2n_k-2}q_2\ldots q_{n_k},a) >0.
\end{equation}
By Theorem~\ref{ins2},
\begin{equation}
\label{L14}
(-1)^{m-1}S_{n_k}^q(a^{2m-3}q_2\ldots q_{m-1} \sqrt{q_{m}},a)\geq 0, \ 
\ m=3,4,\ldots, n_k-1.
\end{equation}
By Theorem~\ref{inn1},
\begin{equation}
\label{L15}
S_{n_k}^q(x_0,a)<0, \ \ \text{where} \ \ x_0 \in (1, a^2).
\end{equation}
It follows from (\ref{L11}), (\ref{L12}), (\ref{L13}), (\ref{L14}), and (\ref{L15}) that 
$S_{n_k}(z)$ has only real positive zeros for all $k=1,2,\ldots.$ Since $\lim_{k\to \infty} 
S_{n_k}(z)= (f \ast g_a )(z),$ the function $f \ast g_a \in \mathcal{LP }I.$
Theorem~\ref{ins3} is proved. 
\end{proof}

Theorem~\ref{Th3} follows from Theorem~\ref{ins3} in the case when there exists a sequence 
of positive integer numbers $(n_k)_{k=1}^{\infty}$ such that $ q_{n_k} \geq \frac{4}{a^2}.$
Now assume that there exists a positive integer $n_0$ such that $q_n<\frac{4}{a^2}$ for all $n\geq n_0.$

\begin{Lemma} \label{ll12} When $1\leq q_n<\frac{4}{a^2},$ the inequality below is true
$$1-\frac{2}{a \sqrt{q_n}}+\frac{1}{a^4 q_n q_{n-1}} -\frac{1}{a^9 q_n^{3/2}q_{n-1}^2} \geq $$
\begin{equation}
\label{c1}
1-\frac{2}{a}+\frac{1}{a^4 q_{n-1} } -\frac{1}{a^9 q_{n-1}^2}. 
\end{equation}
\end{Lemma}
The proof of Lemma~\ref{ll12} is similar to the proof of Lemma~\ref{ll8}.

\begin{Lemma} \label{ll13} 
When $a\geq 1.87152,$ and $q_{n-1}<\frac{4}{a^2},$ the inequality below is true:
$$(-1)^{n-1}S_n^q(a^{2n-3}q_2\ldots q_{n-1}\sqrt{q_n},a) >0.$$
\end{Lemma}

\begin{proof} It is easy to see that the right-hand part of the inequality (\ref{c1}) is 
a decreasing function in $q_{n-1}.$ Thus, taking into account (\ref{L10}) and the 
fact that $q_{n-1}<\frac{4}{a^2},$ we have 
\begin{equation}
\label{c2}
(-1)^{n-1}S_n^q(a^{2n-3}q_2\ldots q_{n-1}\sqrt{q_n},a) \geq 1-\frac{2}{a}+
\frac{1}{4 a^2 } -\frac{1}{16a^5 } >0,
\end{equation}
when  $a\geq 1.87152.$
Lemma~\ref{ll13} is proved. 
\end{proof}
 
In the case when $q_n<\frac{4}{a^2}, \ \ n\geq n_0,$ Theorem~\ref{Th3} follows 
from Lemma~\ref{ll13} and estimations similar to (\ref{L12}), (\ref{L13}), (\ref{L14}), and (\ref{L15}).

Theorem~\ref{Th3} is proved. $\Box$

\section*{Statements and Declarations}
On behalf of all authors, the corresponding author states that there is no conflict of interest.

\end{document}